\documentclass{article}
\usepackage{amsfonts}
\usepackage{amsmath}
\usepackage{amsthm, amscd}
\usepackage{mathtools}
\usepackage{amssymb}
\usepackage{mathrsfs}

\usepackage[alphabetic]{amsrefs}
\usepackage{etex}
\usepackage{hyperref}
\hypersetup{
	colorlinks = true,
	linkcolor  = black
}
\usepackage{color}

\usepackage{tikz}
\usepackage{xypic}

\theoremstyle{plain}\newtheorem{Theorem}{Theorem}[section]
\theoremstyle{plain}
\theoremstyle{plain}\newtheorem{Lemma}[Theorem]{Lemma}
\theoremstyle{plain}\newtheorem{Definition}[Theorem]{Definition}
\theoremstyle{plain}\newtheorem{Proposition}[Theorem]{Proposition}
\theoremstyle{plain}
\theoremstyle{plain}
\theoremstyle{plain}\newtheorem*{Theorem*}{Theorem}

\theoremstyle{remark}\newtheorem{remark}[Theorem]{Remark}

\DeclareMathOperator{\trunk}{trunk}
\DeclareMathOperator{\inte}{int}

\newcommand{\bR}{\mathbb{R}}
\newcommand{\bZ}{\mathbb{Z}}

\newcommand{\frh}{\mathfrak{h}}

\title{The trunk number of satellite knots and Thurston norm}
\author{Zehan Pan}
\begin{document}
\maketitle

\begin{abstract}
	Assume $J\subset \bR ^3$ is a non-trivial knot, and assume $\hat k\subset S^1\times D^2$ is a satellite pattern. Let $N$ be the generalized Thurston norm of the homology class of the meridian disk in $S^1\times D^2$ with respect to $\hat k$.  Let $K$ be the satellite knot of $J$ with pattern $\hat k$.  We show that the trunk number of $K$ is strictly greater than $N$ times the trunk number of $J$.
	
    This paper is written with the help of Prof. Zhenkun Li and Prof. Boyu Zhang.
\end{abstract}

\section{Introduction}

The $trunk\ number$ is a knot invariant defined by Ozawa\cite{Ozawa}. It is closely related to the $width$ of a knot introduced by Gabai\cite{Gabai}. The properties of the trunk number under connected sums were studied by Davies and Zupan\cite{Zupan}. In \cite{Kavi-Wu-Li}, Kavi, Wu, and Li studied the properties of the trunk number of satellite knots: they proved that if $K$ is a satellite knot with companion $J$ and if $J$ is non-trivial, then
$$
trunk(K)>\frac{1}{2}m\cdot trunk(J),
$$
where $m$ denotes the wrapping number of the pattern of $K$. The purpose of this paper is to give another lower bound for the trunk of satellite knots, but using the $generalized\ Thurston\ norm$ (see Section \ref{Statement_result}) instead of wrapping number. Let $N$ be the generalized Thurston norm of the pattern of $K$, and assume $J$ is non-trivial, we show that 

$$
trunk(K)>N\cdot trunk(J)
$$

\subsection{Notation and conventions}
Throughout this paper, all knots, curves, and surfaces are assumed to be smooth, and maps between manifolds are assumed to be smooth.

It will be important for us to differentiate between embedded circles and \emph{isotopy classes} of embedded circles in $S^3$. 
 Therefore, we will  use capitalized letters ($K$, $J$, etc.) to denote isotopy classes of embedded circles (namely, knots) in $S^3$, and use lower case letters ($k$, $j$, etc.) to denote embedded circles in $S^3$.
 
 We will use notations such as $\hat k$, $\hat j$, etc, to denote embedded circles in $S^1\times D^2$. To simplify notation, we also use $\hat V$ to denote $S^1\times D^2$. 
 
 We will identify $S^3$ with the unit sphere $\{(x_1,x_2,x_3,x_4)\in \bR^4|\sum_{i=1}^4 x_i^2 = 1\}$, and define
\begin{center}
 \begin{tabular}{ccc}
	$\frh$:  &$S^3$ &$\to \bR$ \\
 	&$(x_1, x_2, x_3, x_4)$&$\mapsto x_4$
 \end{tabular}
\end{center}
 to be the height function on $S^3$. We call $(0,0,0,1)$, $(0,0,0,-1)\in S^3$ the \emph{north pole} and the \emph{south pole} respectively.
 
 If $X$ is a manifold with finitely many connected components, we use $|X|$ to denote the number of connected components of $X$.

\subsection{Statement of the result}
\label{Statement_result}

Suppose $j\subset \bR ^3$ is a knot such that the restriction of $\frh$ to $j$ is a Morse function with distinct critical values. We also assume that $j$ does not contain the north pole or the south pole of $S^3$. The trunk number of $j$ is defined to be
$$
\trunk(j) = \max_{z} |\frh^{-1}(z)\cap j|.
$$
It is straightforward to see that $\trunk(j)$ is invariant under $C^2$--small perturbations of $j$. 

If $J$ is the isotopy class of a knot, the trunk number of $J$ is defined to be
$$
\trunk(J) = \min_{j} \trunk(j),
$$
where $j$ is taken over all embedded circles in $\bR^3$ in the isotopy class of $J$ such that the restriction of $\frh$ to $j$ is a Morse function with distinct critical values, and that $j$ is disjoint from the north and the south poles.

Now assume $\hat{k}\subset S^1\times D^2$ is a satellite pattern. By definition, this implies that $\hat k$ can not be included in a solid ball in $S^1\times D^2$. 

Recall that we use $\hat V$ to denote $S^1\times D^2$. 
Let $(\Sigma,\partial \Sigma)\subset (\hat V, \partial \hat V)$ be an embedded compact oriented surface that intersects $\hat k$ transversely. If $\Sigma$ is connected, we define 
$$
x_{\hat k}(\Sigma) = \max\{-\chi(\Sigma) + |\Sigma\cap \hat k|,0\}.
$$
If $\Sigma$ is disconnected with connected components $\Sigma_1,\dots, \Sigma_s$, define
$$
x_{\hat k}(\Sigma) = \sum_{i=1}^s x_{\hat k}(\Sigma_i).
$$
For 
$$a\in H_2(\hat V,\partial  \hat V),$$ the \emph{generalized Thurston norm} of $a$ with respect to $\hat k$ is defined to be
$$
x_{\hat{k}} (a) = \min_{\Sigma} x_{\hat k}(\Sigma) ,
$$
where $\Sigma$ is taken over all 
$(\Sigma, \partial \Sigma)\subset (\hat V,  \partial \hat V)$
such that $[\Sigma]=a$ and that $\Sigma$ intersects $\hat k$ transversely. 

We will fix a satellite pattern $\hat k$ from now.
Let $\Sigma$ be an oriented meridian disk in $\hat V$, let $[\Sigma]\in H_2(\hat V, \partial \hat V)$ be its fundamental class, and define
$$
N = x_{\hat k}([\Sigma]).
$$
 It is straightforward to verify that 
 $$x_{\hat k}(a) = x_{\hat k}(-a)
 \quad \text{ for all } \quad
 a\in H_2(S^1\times D^2, S^1\times \partial D^2),$$ 
 so the value of $N$ does not depend on the orientation of $\Sigma$

The main theorem of this paper is the following result.

\begin{Theorem}
	\label{thm_main}
	Assume $J$ is a non-trivial knot, and $K$ is the satellite of $J$ with pattern $\hat k$. Then
	\begin{equation}
		\label{eqn_main_thm}
			\trunk(K) > N\cdot \trunk(J).
	\end{equation}

\end{Theorem}

\begin{remark}
	The assumption that $J$ is non-trivial is necessary in Theorem \ref{thm_main}. In fact, if $J$ is the unknot and $\hat k$ is the Whitehead double pattern, then $N=1$, and $K$ is also the unknot. So \eqref{eqn_main_thm} does not hold for this case.
\end{remark}
	
	\subsection{Acknowledgement}
    The author wishes to thank Professor Zhenkun Li for proposing the theorem, and is also very grateful to Professor Boyu Zhang for many useful conversations and, especially, help on the homology part. The author would also thank PRISMS (Princeton International School of Math and Science) for its research program.
    
\section{Preliminaries}
\label{sec_preliminaries}
Let $\hat k$, $J$, $K$ be as in Theorem \ref{thm_main}. Recall that we use $\hat V$ to denote $S^1\times D^2$. 
Define
$$
\hat j= S^1\times \{(0,0)\} \subset \hat V
$$
to be the core circle in $\hat V$.

Let $k\subset S^3$ be a representative of $K$ such that $\frh|_k$ is Morse with distinct critical values and that $k$ is disjoint from the north and the south poles of $S^3$.

By the definition of satellite knots, there exists an embedding
\begin{equation}
	\label{eqn_defn_tau}
\tau :\hat V \to S^3
\end{equation}
such that $\tau(\hat j)$ represents the isotopy class $J$ and $\tau(\hat k) = k$. Moreover, we may choose $\tau$ so that its image is disjoint from the north and the south poles.  

Let 
$$V=\tau(\hat V),\quad T = \tau (\partial \hat V).$$ 
Then $V$ is an embedded solid torus in $S^3$, and $T$ is its boundary torus. By the above assumptions, both $V$ and $T$ are disjoint from the north and the south poles.

Perturb $\tau$ near $\partial \hat V$ so that $\frh|_{T}$ is Morse with distinct critical values, and that the critical values of $\frh|_T$ are disjoint from the critical values of $\frh|_k$. Assume 
$$
c_1<c_2<\dots<c_n
$$
are all the critical values of $\frh|_T$ and $\frh|_k$. For each $i=1,2,\dots,n-1$, take $r_i\in (c_i,c_{i+1})$. Then $r_i$ is a regular value of both $\frh|_T$ and $\frh|_k$, and we have 
$$
\trunk(k) = \max_{1\le i \le n-1} |\frh^{-1}(r_i) \cap k|.
$$

Note that $H_2(V,T)= H_2(V,\partial V)\cong \bZ$. We have the following elementary properties about embedded surfaces in $V$.
\begin{Lemma}
	\label{lem_connected_surface_fundamental_class}
	Assume $(\Sigma,\partial \Sigma)\subset (V,T)$ is an embedded connected surface with boundary. Then $[\Sigma]\in H_2(V,T)$ is either zero or a generator of $H_2(V,T)$.
\end{Lemma}

\begin{proof}
	Let $N(\Sigma)$ be a tubular neighborhood of $\Sigma$. Then for each $x\in V\backslash \Sigma$, there is a path in $V\backslash \Sigma$ from $x$ to a point in $N(\Sigma)\backslash \Sigma$. Since $N(\Sigma)\backslash \Sigma$ has two connected components, the space $V\backslash \Sigma$ has at most two connected components.
	
	If $V\backslash \Sigma$ has two connected components, write these components as $M_1$ and $M_2$. Then 
	$$
	[\Sigma] = [\partial M_1] = 0 \in H_2(V,T).
	$$
	
	If $V\backslash \Sigma$ has one connected component, then there exists an embedded circle in the interior of $V$ that intersects $\Sigma$ transversely at one point. Therefore $[\Sigma]$ must be a generator of  $H_2(V,T)$.
\end{proof}

\begin{Definition}
	We say that a simple closed curve $c\subset T$ is \emph{essential} on $T$, if it is not contractible on $T$.  Otherwise, we say that $c$ is \emph{inessential}. 
\end{Definition}

\begin{Lemma}
	\label{lem_nonzero_fundamental_class_odd_components}
	Assume $(\Sigma,\partial \Sigma)\subset (V,T)$ is an embedded connected surface with boundary. Then $[\Sigma]\in H_2(V, T)$ is a generator of $H_2(V,T)$ if and only if there is an odd number of connected  components of $\partial \Sigma$ that are essential on $T$.
\end{Lemma}

\begin{proof}
	Let $D$ be an oriented meridian of $V$. By Lemma \ref{lem_connected_surface_fundamental_class} and the homology exact sequence for the pair $(V,T)$, we have 
	\begin{enumerate}
		\item $[\partial \Sigma] \in H_1(T) $ is equal to $0$ or $ \pm [\partial D]$, 
		\item $[\Sigma]$ is a generator of $H_2(V,T)$ if and only if $[\partial \Sigma]  = \pm [\partial D]$.
	\end{enumerate}
Since the essential components of $\partial \Sigma$ are all parallel to each other, we conclude that $[\partial \Sigma]  = \pm [\partial D]$ if and only if $\partial \Sigma$ has an odd number of essential components. Hence the result is proved.
\end{proof}

\begin{Lemma}
	\label{lem_essential_Thurston_norm_bound}
	Assume $(\Sigma,\partial \Sigma)\subset(V,T)$ is an embedded connected oriented surface such that $[\Sigma]$ is a generator of $H_2(V,T)$. Also assume that $\Sigma$ intersects $k$ transversely. Let $a$ be the number of connected components of $\partial \Sigma$ that are essential on $T$. Let $b = |k\cap \Sigma|$. Then we have
	$$
	\max\{a+b-2, 0\}\ge N.
	$$
\end{Lemma}

\begin{proof}
	Let $c_1,\dots,c_s$ be all the connected components of $\partial \Sigma$ that are inessential. Then each $c_i$ bounds an embedded disk on $T$. By a standard innermost disk argument, we find an embedded oriented surface
	$$
	(\Sigma',\partial \Sigma')\subset(V,T),
	$$
	such that 
	\begin{enumerate}
		\item $\partial \Sigma' = \partial \Sigma \backslash (c_1\cup c_2\cup \dots \cup c_s)$,
		\item $\Sigma'$ intersects $k$ transversely,
		\item $k\cap \Sigma = k\cap \Sigma'$,
		\item $[\Sigma'] =[\Sigma]\in H_2(V,T)$. 
	\end{enumerate}
As a consequence, we have 
$$
	\max\{a+b-2, 0\} = x_{\hat{k}} (\tau^{-1}(\Sigma')) \ge N,
$$
where $\tau$ is the map defined in \eqref{eqn_defn_tau}. 
\end{proof}

We will use the following result by \cite{Guo-Li} and \cite{Zupan}.

\begin{Theorem}
\label{thm_circle_l_in_V}
(\cite{Guo-Li}, see also \cite{Kavi-Wu-Li} Proposition 3.3)
There exists an embedded circle $l\subset V$, such that the following holds. 
\begin{enumerate}
	\item The circle $l$ is transverse to $\frh^{-1}(r_i)$ for all $i=1,\dots,n-1$.
	\item Let $L$ be the isotopy class of $l$. Then there exists an isotopy class  $J'$, such that $L=J\# J'$.
	\item If $\Sigma$ is a component of $\frh^{-1}(r_i)\cap V$ such that $[\Sigma]=0\in H_2(V,T)$, then $l$ is disjoint from $\Sigma$. 
	\item If $\Sigma$ is a component of $\frh^{-1}(r_i)\cap V$ such that $[\Sigma]$ is a generator of $H_2(V,T)$, then $l$ intersects $\Sigma$ transversely at one point.
\end{enumerate}
\end{Theorem} 

\begin{proof}
The construction of $l$ is given in Definition 3.4 of \cite{Guo-Li}. The fact that $J$ is a connected sum component of $L$ is given in the proof of Lemma 4.3 in \cite{Guo-Li}. The discussion of \cite{Guo-Li} after Lemma 4.4 showed that $l$ can be isotoped in $T$ so that it satisfies Conditions 1,3,4 above.
\end{proof}

\begin{Theorem}[\cite{Zupan}]
\label{thm_knot_sum}
Suppose $K_1, K_2$ are two knots, then 
$$
\trunk(K_1\# K_2) = \max\{\trunk(K_1), \trunk(K_2)\}.
$$
\end{Theorem}

\section{Proof of the main result}

Now we prove Theorem \ref{thm_main}.  We will use the same notation from Section \ref{sec_preliminaries}.

For each $i=1,\dots,n-1$, we have $\partial V\cap \frh^{-1}(r_i)\neq\emptyset$ and the intersection is transverse. 

Let $\Sigma$ be a connected component of $\frh^{-1}(r_i)\cap V$. To simplify notation, we say that a connected component of $\partial \Sigma$ is \emph{essential}, if it is an essential curve on $T$; otherwise we say that the boundary component is \emph{inessential}. We say that $\Sigma$ is \emph{principal}, if the fundamental class $[\Sigma]\in H_2(V,T)$ is nonzero. By Lemma \ref{lem_nonzero_fundamental_class_odd_components}, $\Sigma$ is principal if and only if $\partial \Sigma$ has an odd number of essential components.

\begin{Lemma}
	\label{lem_component_in_Dc}
	Let $\Sigma$ be a connected component of $\frh^{-1}(r_i)\cap V$, and assume $c$ is an essential connected component of $\partial\Sigma$. Let $D_c$ be the disk bounded by $c$ on the sphere  $\frh^{-1}(r_i)$ whose interior is disjoint from $\Sigma$. Then there exists at least one component $\Sigma'$ of $\frh^{-1}(r_i)\cap V$ that is distinct from $\Sigma$, such that 
	\begin{enumerate}
		\item $\Sigma'\subset D_c$.
		\item $\Sigma'$ has at least one essential boundary.
	\end{enumerate}
\end{Lemma}

\begin{proof}
Assume the contrary, then every component of $\inte(D_c)\cap T$ is an inessential simple closed curve on $T$, where $\inte(D_c)$ denotes the interior of $D_c$.
Let $S$ be the set of all possible embedded disks $D$ in $S^3$ such that 
\begin{enumerate}
	\item $\partial D = \partial D_c = c$, 
	\item There exists an open neighborhood $N(c)$ of $c$ such that $N(c)\cap D_c = N(c) \cap D$,
	\item $\inte(D)$ intersects $T$ transversely, and every component of $\inte(D)\cap T$ is inessential on $T$,
\end{enumerate}
Since $D= D_c$ satisfies the above conditions, the set $S$ is non-empty. Let $D$ be an embedded disk in $S$ that minimizes the value of $|\inte(D)\cap T|$.  

We show that $\inte(D)\cap T = \emptyset$. Assume the contrary, then by the assumptions, every connected component of $\inte(D)\cap T$ bounds a disk in $T$. We may take a component $c'$ such that it bounds an ``innermost disk''.  Namely,  $c'$ is a connected component of $\inte(D)\cap T $, and it bounds a disk $B_T$ in $T$ such that $\inte(B_T)\cap \inte(D)=\emptyset$. Since $c'$ is a simple closed curve on $D$, it bounds a disk $B_D$ in $D$. Perturbing  $(D\backslash B_D)\cup B_T$ yields an element $D'$ in $S$ such that $|\inte(D')\cap T|<|\inte(D)\cap T|$, which contradicts the definition of $D$. 

Since $\inte(D)\cap T = \emptyset$ and $N(c)\cap D_c = N(c) \cap D$ for some neighborhood $N(c)$ of $c$, we must have $\inte(D)\cap V = \emptyset$.

In conclusion, we proved that $c$ bounds a disk in $S^3$ whose interior is disjoint from $V$. Since $c$ is essential on $T$, this implies that $j$ is isotopic to the unknot, which contradicts the assumption on the non-triviality of $J$.
\end{proof}

We will need the following technical combinatorial lemma about surfaces on a sphere.
\begin{Lemma}
	\label{lem_combinatorics_on_sphere}
	Assume $s$ is a positive integer and $\Sigma_1,\dots,\Sigma_s$ are disjointly embedded compact surfaces in $S^2$. Assume for all $j=1,\dots,s$, each connected component of $\partial \Sigma_j$ is classified as either ``essential'' or ``inessential''. Let $n_j$ be the number of essential boundary components of $\Sigma_j$. 
	
	Assume that for every $j$ and every essential component $c$ of $\Sigma_j$, there exists $j'\neq j$ such that $n_{j'}>0$ and $\Sigma_{j'}\subset D_c$, where $D_c$ is the disk bounded by $c$ on $S^2$ whose interior is disjoint from $\Sigma_j$.
	
	Also assume that there exists at least one $j$ with $n_j>0$.
	
	Then
	$$
	\sum_{n_j \text{ is odd}} (2-n_j) > 0.
	$$
\end{Lemma}

\begin{proof}
	Regarding essential boundaries, there are three cases:
 \begin{enumerate}
     \item All boundaries are essential.
    In this case, we assert that every component $\Sigma_j$ has at least 1 boundary. If this is untrue, then there exists such $j$ that $n_j=0$. Because $\Sigma_j\subset S^2$, $\Sigma_j$ must cover the entire $S^2$, and since the components are disjoint, there can be no other components and thus the conclusion is trivial.

    Now for every even $n_j$, we have $n_j\ge 2$, so $2-n_j\le 0$.

    Let the connected components of $S^2/(\Sigma_1\cup...\cup\Sigma_s)$ be $T_1,\dots,T_k$, and we define these components as "gaps". Now for any gap $T_k$, it must have at least 2 boundaries. Otherwise, $T_k$ will be a disk whose interior does not have any other components, which contradicts the combinatorial setup. 

    Notice that every component and gap lie on $S^2$, so each of them has genus 0, which means each has an Euler characteristic, or $\chi$ value, of $2-b$, with $b$ being the number of boundaries. Thus, the $\chi$ value of every gap is no greater than 0. 

    Now we sum up the $\chi$ values of all the surfaces on $S^2$. Since $\chi(A\cup B)=\chi(A)+\chi(B)-\chi(A\cap B)$, and the intersection between any component and any gap is a loop (whose $\chi$ value is 0) if there exists such intersection, we have:
    
    \begin{center}
    $$
    \sum_{n_j\ is\ even} \chi(\Sigma_j) + \sum_{n_j\ is\ odd} \chi(\Sigma_j) + \sum \chi(T_k) = \chi(sphere) = 2.
    $$
    \end{center}
    Since all the boundaries are essential, $2-n_j=\chi(\Sigma_j)$. And we know that $2-n_j\le 0$ for even $n_j$, so 
    $$
	\sum_{n_j\ is\ even} \chi(\Sigma_j)\le 0.
    $$ Also, since $\chi(T_k)\le 0$ for every $k$, we have 
	$$
    \sum \chi(T_k)\le 0.
    $$ Therefore,$$
    \sum_{n_j\ is\ odd} \chi(\Sigma_j)\ge 2.
    $$
    which suffices to show that $$
    \sum_{n_j\ is\ odd} \chi(\Sigma_j)>0.
    $$
     \item There is at least one inessential boundary, and each connected component has at least 1 essential boundary.

    In this case, if a gap has only 1 boundary, then it must be inessential, or it will contradict the combinatorial setup. We call these gaps with only 1 boundary "disk gaps".

    Assume there are $d$ disk gaps. Since other gaps have at least 2 boundaries each, we have $\sum \chi(T_k)\le d$. In this case, we can still sum up the Euler characteristics of all the pieces and gaps on $S^2$:

    $$\sum_{n_j\ is\ even} \chi(\Sigma_j) + \sum_{n_j\ is\ odd} \chi(\Sigma_j) + \sum \chi(T_k) = \chi(sphere) = 2.$$
    
    Thus, $$\sum_{all}\chi(\Sigma_j)\ge 2-d.$$ Let $b$ be the total number of inessential boundaries in the pieces, then $$\sum_{all}(2-n_j)=\sum_{all}\chi(\Sigma_j)+b.$$
     
     Every disk cap has a boundary that belongs to a piece, and they sum up to $d$ inessential boundaries. Thus, $$\sum_{all}(2-n_j)\ge \sum\chi(\Sigma_j)+b\ge (2-d)+d=2.$$

     Since every component has at least 1 essential boundary, for $n_j$ even, $\Sigma_j$ has at least 2 essential boundaries. Thus, $$\sum_{n_j\ is\ even}(2-n_j)\le 0$$.

     Hence, $$\sum_{n_j\ is\ odd}(2-n_j)\ge 2-0=2>0.$$

     \item There is at least one inessential boundary, and there is at least 1 component with no essential boundaries.

     Let $t$ be the number of components without essential boundaries.
     
     The combinatorial setup states that any essential boundary bounds a disk disjoint from its component, and the disk has a component with at least 1 essential boundary in it. Thus, if we remove a component with 0 inessential boundaries, the presumption is still correct. Also, $$\sum_{n_j\ is\ odd}(2-n_j)$$ does not change because it only concerns components with an odd number of essential boundaries.

     Removing 1 such component, we have $t-1$ such components left. We can remove them one by one, and eventually the case will become the same as case 2, in which $$\sum_{n_j\ is\ odd}(2-n_j)>0$$ is true. Thus, this statement holds for case 3.
     
 \end{enumerate}
\end{proof}

\begin{Proposition}
	\label{prop_intersection_bound_on_ri}
	For each $i=1,\dots,n-1$, let $z_i$ be the number of principal components of $\frh^{-1}(r_i)\cap V$. If $N>0, z_i>0$, we have
	$$
	|\frh^{-1}(r_i) \cap k| > N \cdot z_i.
	$$
\end{Proposition}
	
\begin{proof}
	Let $\Sigma_1,\dots, \Sigma_s$ be the connected components of $\frh^{-1}(r_i)\cap V$. Let $n_j$ be the number of essential components of $\Sigma_j$.
	Then by Lemma \ref{lem_nonzero_fundamental_class_odd_components}, Lemma \ref{lem_essential_Thurston_norm_bound}, and Lemma \ref{lem_component_in_Dc}, the following statements hold for each $j=1,\dots,s$. 
	\begin{enumerate}
		\item $\Sigma_j$ is principal if and only if it has an odd number of essential boundaries. 
		\item If $\Sigma_j$ is principal, then 
		\begin{equation}
			\label{eqn_in_prop_intersection_bound_Sigma_j_norm}
		\max\{ |\Sigma_j\cap k| + n_j -2, 0\} \ge N.
		\end{equation}
		\item If $c$ is an essential boundary component of $\Sigma_j$, and $D_c$ is the disk bounded by $c$ on $\frh^{-1}(r_i)$ whose interior is disjoint from $\Sigma_j$, then there exists $j'\neq j$, such that 
		\begin{enumerate}
			\item $\Sigma_{j'}\subset D_c$,
			\item $\Sigma_{j'}$ has at least one essential boundary component.
		\end{enumerate}
	\end{enumerate}

For each $j$ such that $n_j$ is odd, we have $\Sigma_j$ is principal. Since we assume $N>0$, inequality \eqref{eqn_in_prop_intersection_bound_Sigma_j_norm} above implies that
$$
|\Sigma_j\cap k|+n_j-2\ge N.
$$ 
By Lemma \ref{lem_combinatorics_on_sphere}, we have
$$
\sum_{{n_j} \text{ is odd}} (2-n_j) >0.
$$
Therefore,
$$
	|\frh^{-1}(r_i) \cap k| \ge \sum_{{n_j} \text{ is odd}}|\Sigma_j\cap k| \ge \sum_{{n_j} \text{ is odd}}(N+2-n_j) >  \sum_{{n_j} \text{ is odd}} N = N\cdot z_i.
$$
\end{proof}	

Now we can prove Theorem \ref{thm_main} using Proposition \ref{prop_intersection_bound_on_ri}. 

\begin{proof}[Proof of Theorem \ref{thm_main}]
	The statement of the theorem is obvious if $N=0$. From now, we assume $N>0$. 
	
	Recall that the embedded circle $l\subset V$ satisfies the statements of Theorem \ref{thm_circle_l_in_V}.
By definition, there exists $i\in\{1,\dots,n-1\}$ such that 
$$|\mathfrak{h}^{-1}(r_i)\cap l| = \trunk(l)>0.$$
By Proposition \ref{prop_intersection_bound_on_ri} and Parts (3), (4) of Theorem \ref{thm_circle_l_in_V}, we have
$$
|\mathfrak{h}^{-1}(r_i)\cap k| > N\cdot |\mathfrak{h}^{-1}(r_i)\cap l| = N\cdot \trunk(l).
$$
By the definition of trunk number, we have
$$
\trunk(k) \ge |\mathfrak{h}^{-1}(r_i)\cap k|.
$$
Recall that $L$ denotes the isotopy class of $l$.  By Theorem \ref{thm_circle_l_in_V} Part (2) and Theorem \ref{thm_knot_sum}, we have
$$
\trunk(l)\ge \trunk(L)\ge \trunk(J).
$$
Therefore, combining the above inequalities, we have
\begin{equation}
	\label{eqn_lower_bound_trunk_k}
\trunk(k)> N\cdot \trunk(J).
\end{equation}
Since \eqref{eqn_lower_bound_trunk_k} holds for every embedded circle $k$ representing the isotopy class of $K$, we conclude that $\trunk(K) > N\cdot \trunk(J)$. 
\end{proof}

\end{document}